\documentclass[11pt]{article}
\usepackage{latexsym}
\usepackage{comment}
\usepackage{amsmath,amsthm,amsfonts,amssymb,graphicx,epsfig,latexsym,color}
\usepackage{epsf,enumerate}
\usepackage{fancybox}
\usepackage{tikz-cd}
\usetikzlibrary{arrows}

\title{Homological stability of $Aut(F_n)$ revisited}
\author{Mladen Bestvina\thanks{The
    author gratefully acknowledges the support by the National
    Science Foundation under grant number 1308178.}}

\date{February 27, 2015}

\newtheorem{thm}{Theorem}
\newtheorem{lemma}[thm]{Lemma}

\newtheorem{prop}[thm]{Proposition}
{}
{}
\newtheorem{question}[thm]{Question}

\theoremstyle{remark}

\newtheorem{remark}[thm]{Remark}

\newtheorem*{definition*}{Definition}
\newtheorem*{remark*}{Remark}

\def\Z{{\mathbb Z}}

\def\Q{{\mathbb Q}}

\renewcommand{\>}{\rangle}

\def\A{{\mathbb A}}
\begin{document}

\maketitle
\begin{abstract} We give another proof of a theorem of Hatcher and Vogtmann
  stating that the sequence $Aut(F_n)$ satisfies integral homological
  stability. The paper is for the most part expository, and we also
  explain Quillen's method for proving homological stability.\end{abstract}

\section{Introduction}

Let $G_1\subset G_2\subset G_3\subset\cdots$ be a sequence of
groups. For example, $G_n$ could be any of the following:
the permutation group $S_n$, or the signed permutation group $S_n^{\pm}$,
braid group $B_n$,
$SL_n(\Z)$,
$Aut(F_n)$,
and many other groups, with all inclusions standard. The sequence
satisfies {\it homological stability} if for every $r$ there is $n(r)$
such that for $n\geq n(r)$ inclusion induced $H_r(G_n)\to
H_r(G_{n+1})$ is an isomorphism. All of the above sequences satisfy
homological stability.

Homological stability of $Aut(F_n)$ over $\Q$ was proved by Hatcher
and Vogtmann by a very elegant argument \cite{MR1678155}, as
follows. First, they show that $Aut(F_n)$ acts properly on an $r$-connected
simplicial complex $S\A_{n,r+1}$, and second, that for $n>2r$ the quotient
spaces $Q_{n,r+1}=S\A_{n,r+1}/Aut(F_n)$ and $Q_{n+1,r+1}$ are {\it canonically
  homeomorphic}. Since
$$H_r(Aut(F_n);\Q)=H_r(Q_{n,r+1};\Q)=H_r(Q_{n+1,r+1};\Q)=H_r(Aut(F_{n+1};\Q))$$
stability follows.

This is a very transparent reason for stability, and I am not aware of
any other example where stability can be proved in this way.

\begin{question} Can one prove rational homological stability for
  braid groups or mapping class groups in the same way?
\end{question}

Integral stability of $Aut(F_n)$ is more subtle. It was first
established by Hatcher and Vogtmann in \cite{MR1678155} by studying a
spectral sequence associated to the action of $Aut(F_n)$ on the
complex of ``split factorizations'' of $F_n$. Further, it is known, by
the work of Hatcher, Vogtmann and Wahl, that $Aut(F_n)\to Out(F_n)$
induces an isomorphism in $H_r$ when $n\geq 2r+4$, see
\cite{MR1314940,MR1678155,MR1671188,MR2113904,MR2220689}. The proof is
based on Quillen's method and requires a rather delicate spectral
sequence argument.

In this note we give a proof of integral stability in the same spirit
as Hatcher-Vogtmann's proof of rational stability. We view the
quotient spaces $Q_{n,r+1}$ as {\it orbi spaces} and observe that for
$n>>r$ and canonical identification $Q_{n,r+1}=Q_{n+1,r+1}$ of
underlying topological spaces, the groups $\Gamma_{n,r+1}(x),\Gamma_{n+1,r+1}(x)$
associated to a point $x$ (i.e. stabilizers of corresponding points in
$S\A_{n,r+1},S\A_{n+1,r+1}$) themselves belong to a sequence of finite
groups satisfying homological stability. More precisely,
$\Gamma_{n,r}(x)=G_r(x)\times S^{\pm}_{n-2r}$ where $G_r(x)$ does not depend
on $n$, and $S^{\pm}_i$ is the signed permutation group on $i$
symbols. Integral stability of $Aut(F_n)$ easily follows. 

We emphasize here that we will use spectral sequences only to prove
homological stability for signed permutation groups. The rest of the
argument is geometric, in the spirit of Hatcher-Vogtmann
\cite{MR1678155}. 

The price we must pay for conceptual transparency is that our
stability range is far from optimal. The argument seems to require
$n>4r$ (while the best known estimate is $n\geq 2r+2$), although it is
possible that this may be improved with further effort. 

We note that
Galatius \cite{MR2784914} computed stable homology groups. 
For a more systematic approach to homological stability of
automorphism groups see \cite{wahl}.

{\bf Outline.} In Section \ref{how} we recall Quillen's method for
proving stability. We will only need the simple form where the group
acts on a highly connected space with one orbit of cells in a
dimension range. We then prove homological stability for permutation groups,
following an argument of Maazen, and we give a variant for signed
permutation groups. The latter groups naturally appear as subgroups of
$Aut(F_n)$ that act as symmetry groups of a rose. The final section
elaborates on the outline given above. Instead of working with orbi
spaces $X/G$, we use the Borel construction and consider $X\times_G
EG=(X\times EG)/G$ where $G$ acts diagonally on $X\times EG$. This is
technically more expedient, but the reader should keep the orbi space
picture in mind.

{\bf Acknowledgments.} I thank the referee for useful comments. I also
thank the organizers of MSJ-SI for giving me a nudge. I gave talks on
this proof some years ago but didn't get around to writing it up until
now.

\section{Quillen's method}\label{how}

Here we describe a method, due to Quillen (unpublished), to prove homological
stability. For published accounts of Quillen's method see
e.g. \cite{MR0472796,MR586429}. 

We will say that $X$ is {\it $r$-connected} if $\tilde H_i(X)=0$ for
$i\leq r$.

Fix a sequence $G_0\subset G_1\subset G_2\subset\cdots$ of groups.
Suppose that $H_i(G_{s-1})\to H_i(G_{s})$ is an isomorphism when
the following hold:
\begin{itemize}
\item $i=r-1$ and $s\geq n-2$, or
\item $i<r-1$ and $s\geq n+i-r-2$.
\end{itemize}

Also assume that the group $G_n$ acts on a $\Delta$-complex $X=X(n)$
of dimension $\leq n-1$ with the following properties.
\begin{enumerate}[(i)]
\item Action is without inversions, i.e. any element that leaves a
  simplex invariant fixes it pointwise.
\item $X$ is $r$-connected.
\item in each dimension $0,1,\cdots,r$ there is one orbit of
  simplices.
\item There is a flag of
  simplices $$\sigma^0<\sigma^1<\cdots<\sigma^r$$ in $X$ 
  such that
$Stab(\sigma^i)=G_{n-i-1}\subset G_n$.
\item If $\tau_1^{i}$ and $\tau_2^{i}$ are two $i$-simplices
  in $X$ contained in $\rho^{i+1}$ as faces ($i=0,1,\cdots,r-1$) then
  there exists $g\in G$ such that:
\begin{itemize}
\item $g(\tau_1)=\tau_2$, and
\item $g$ commutes with all elements of $Stab(\rho)$.
\end{itemize}
\end{enumerate}

\begin{remark}
In view of (iii) and (iv), we have that the stabilizer of {\it every}
$i$-simplex is a conjugate of $G_{n-i-1}$.  Note that conjugation
induces identity in the homology of a group. So we have a canonical
isomorphism $H_*(Stab(\tau^i))\cong H_*(Stab(\sigma^i))=H_*(G_{n-i-1})$ for
any $i$-simplex $\tau$ (by choosing any $g\in G$ with $g(\tau)=\sigma$
and passing to the isomorphism in homology induced by conjugation
$g_*:Stab(\tau)\to Stab(\sigma)$, $h\mapsto ghg^{-1}$; this isomorphism is
independent of the choice of $g$).

Property (v) guarantees that stabilizer inclusions
$Stab(\rho)\hookrightarrow Stab(\tau_j)$, $j=1,2$, induce the {\it
  same} homomorphism $H_*(G_{n-i-2})\to H_*(G_{n-i-1})$ in homology,
after the identifications in the previous paragraph. This follows by
considering the diagram

\begin{center}
\begin{tikzcd}
  Stab(\rho) \arrow[hookrightarrow]{rd} \arrow[hookrightarrow]{r} & Stab(\tau_1) \arrow{d}{g_*}\\
                                 & Stab(\tau_2)
\end{tikzcd}
\end{center}

\noindent
which commutes (by the assumption that $g_*$ fixes $Stab(\rho)$) and
passing to homology. 
By (iv) this is the homomorphism induced by inclusion
$G_{n-i-2}\hookrightarrow G_{n-i-1}$.
\end{remark}

\begin{prop}\label{quillen} 
Under the above assumptions
$H_r(G_{n-1})\to H_r(G_n)$ is an isomorphism.
\end{prop}

\begin{proof}
Consider the ``equivariant homology spectral sequence''. This is the
spectral sequence associated to the filtration of $Y=X\times_{G_n}
EG_n$ coming from the skeleta of $X$: $Y_p=X^p\times_{G_n}
EG_n$. Since $X$ is $r$-connected we have $H_r(Y)=H_r(G_n)$. The first
page
is $$E^1_{p,q}=H_{p+q}(Y_p,Y_{p-1})=H_q(Stab(\sigma^p))=H_q(G_{n-p-1})$$
with equalities legal since identifications are up to inner
automorphisms, which induce identity in homology. When $p$ is even the
differential $E^1_{p+1,q}\to E^1_{p,q}$ is 0 by (v) (a $(p+1)$-simplex
has an even number of $p$-faces, stabilizer inclusions are all standard,
half come with positive and half with negative sign). In particular,
$E^1_{0,q}=H_q(G_{n-1})$ survives to $E^2$. Likewise, when $p$ is odd
the differential $E^1_{p+1,q}\to E^1_{p,q}$ is the inclusion induced
$H_q(G_{n-p-2})\to H_q(G_{n-p-1})$. A portion of the first page is
pictured below. The leftmost column corresponds to $p=0$ and the top
row to $q=r$. 

\[\arraycolsep=1.4pt\def\arraystretch{2.2}
\begin{array}{ccccccccc}
H_r(G_{n-1})&\overset{0}\leftarrow & H_r(G_{n-2})&\leftarrow &
H_r(G_{n-3})& \\ H_{r-1}(G_{n-1})&\overset{0}\leftarrow &
\Ovalbox{$H_{r-1}(G_{n-2})$}&\overset{\cong}\leftarrow &
\Ovalbox{$H_{r-1}(G_{n-3})$}&\overset{0}\leftarrow\\ H_{r-2}(G_{n-1})&\overset{0}\leftarrow
& H_{r-2}(G_{n-2})&\overset{\cong}\leftarrow &
\Ovalbox{$H_{r-2}(G_{n-3})$}&\overset{0}\leftarrow &
\Ovalbox{$H_{r-2}(G_{n-4})$}&\overset{\cong}\leftarrow & H_{r-2}(G_{n-5})
\end{array}
\]

The circled terms are $E^1_{p,r-p}$ and $E^1_{p,r-p+1}$ with $p>0$ and
$q<r$, and they die thanks to our assumption that the differential
$d^1:E^1_{p+1,q}\to E^1_{p,q}$ is an isomorphism when $p$ is odd and
it involves at least one of the circled terms. It now follows that the
$E^2$ page has $E^2_{0,r}=H_r(G_{n-1})$ and all terms on the diagonals
$p+q=r$ and $p+q=r+1$ with $q<r$ vanish. Thus the same holds for
the $E^\infty$ page, and since the spectral sequence converges to
$H_{p+q}(Y)$ we have that $H_r(G_{n-1})=E^\infty_{0,r}\to H_r(Y)=H_r(G_n)$ is
an isomorphism, thus proving the Proposition.
\end{proof}

\subsection{Stability for (signed) permutation groups}
Stability for symmetric groups was established by Nakaoka
\cite{MR0112134,MR0131874}. We will follow Maazen's proof \cite{maazen}.
See also \cite{MR2155519}. We start with a couple of elementary
lemmas.

Call a nonempty polyhedron $X$ {\it $n$-spherical} if $\tilde
H_i(X)=0$ for $i\neq n$.

\begin{lemma}\label{2}
Let $X$ be a polyhedron and $X_i\subset X$ a finite collection of $m$
subpolyhedra covering $X$, and let $n\geq 0$ be an integer. Suppose
that for every $k=1,2,\cdots,m$ any $k$-fold intersection
$X_{i_1}\cap\cdots\cap X_{i_k}$ is $(n-k+1)$-spherical (empty for
$k\geq n+2$) whenever $i_1<i_2<\cdots<i_k$. Then $X$ is $n$-spherical.
\end{lemma}

\begin{proof} Induction on $m$. If $m=1$ there is nothing to prove,
  and if $m=2$ the statement follows from 
  Mayer-Vietoris. Assume $m>2$. Let $Y=X_2\cup X_3\cup\cdots\cup
  X_m$. By induction $Y$ is 
  $n$-spherical and $$X_1\cap Y=(X_1\cap X_2)\cup (X_1\cap
  X_3)\cup\cdots\cup (X_1\cap X_m)$$ is 
$(n-1)$-spherical (again by induction). Thus the claim
  follows. 
\end{proof}

\begin{lemma} \label{3}
If $X$ is $n$-spherical and $F$ is a nonempty finite set
  then the join $X*F$ is $(n+1)$-spherical.
\end{lemma}

\begin{proof}
By induction on the cardinality of $F$. If $|F|=1$ then $X*F$ is
contractible. If $|F|=2$ then $X*F$ is the suspension and $\tilde
H_{i+1}(X*F)=\tilde H_i(X)$ so the claim follows. When $|F|>2$ write
$X*F$ as the union of two sets whose intersection is $X$, with one set
contractible and the other $(n+1)$-spherical by induction (join of $X$
and the set $F$ without one of the points). Then use Mayer-Vietoris.
\end{proof}

\begin{prop}
The sequence of symmetric groups $S_n$ satisfies homological
stability:
$$H_r(S_{n-1})\to H_r(S_n)$$ is an isomorphism for $n>2r$.
\end{prop}

\begin{proof}
We argue by induction on $r$, starting with the obvious $r=0$. We will
apply Proposition \ref{quillen} with $G_n=S_n$. Note
that the fact that $H_i(S_{s-1})\to H_i(S_s)$ is an isomorphism when
$i=r-1$ and $s\geq n-2$, or when $i<r-1$ and $s\geq n+i-r-2$ follows
by induction (for the latter case the calculation is $s\geq
n+i-r-2>2r+i-r-2=r+i-2\geq 2i$). 
Consider the $\Delta$-complex $X=X(n)$ whose
vertices are $1,2,\cdots,n$ and there is an (ordered) simplex for
every ordered subset of the vertex set. So e.g. there are $n(n-1)$
edges etc. 

{\bf Claim.} $X(n)$ is $(n-1)$-spherical.

{\bf Proof of Claim.} Let $X_i$ be the union of (closed) simplices in
$X$ whose first vertex is $i$. Thus $X_i$ contains all simplices whose
vertices don't include $i$ (and it is the cone on this
subcomplex). More generally, the $k$-fold intersection of the $X_i$'s
can be identified with the subcomplex (the ``base'') consisting of
simplices that don't involve $k$ particular vertices, with $k$ cones
attached to it. Since the base is a copy of $X(n-k)$ it is
$(n-k-1)$-spherical by induction. It follows from Lemma \ref{3} that
$k$-fold intersections are $(n-k)$-spherical. Now Lemma \ref{2}
implies that
$X=X(n)$ is $(n-1)$-spherical.

The verification of (i)-(v) is left to the reader. Thus stability follows.
\end{proof}

There is an identical proof for the signed permutation group $S_n^\pm$
(or the
hyperoctahedral group), i.e. the Coxeter group of type $B_n$. Recall
that $S_n^\pm$ is the semi-direct product $S_n\ltimes \Z_2^n$ 
and it can be viewed as the group of permutations $\pi$ of the set
$\{-n,-(n-1),\cdots,-1,1,\cdots,n-1,n\}$ such that $\pi(-x)=-\pi(x)$
for all $x$.

\begin{prop}[\cite{MR2736166}]\label{6}
The signed permutation groups satisfy homological stability:
$$H_r(S_{n-1}^\pm)\to H_r(S_n^\pm)$$ is an isomorphism for $n>2r$.
\end{prop}

\begin{proof}
Now let $X=X(n)$ be the $\Delta$-complex with vertex set
$-n,-(n-1),\cdots,-1,1,2,\cdots,n$ and a simplex is an ordered subset
with distinct absolute values. The proof that $X$ is $(n-1)$-spherical
is the same: take $X_i$ to consist of simplices that start with $i$ or
$-i$. 
\end{proof}

\begin{remark}
There is one more infinite sequence of Weyl groups, namely
of type $D_n$. This is the group of signed permutations with an even
number of negative signs, and it has index 2 in $S_n^{\pm}$. One can
prove homological stability as above, by considering the action on the
same complex as for $S_n^{\pm}$ (there are now two orbits of simplices
in the top dimension). For a generalization of this method to other
sequences of Coxeter groups see \cite{Coxeter}.
\end{remark}

\section{Integral homological stability of $Aut(F_n)$}

The following is well-known.

\begin{prop}\label{1}
Let $f:X\to Y$ be a map between spaces equipped with filtrations
$$\emptyset=X_{-1}\subset X_0\subset X_1\subset\cdots\subset X_m=X$$
and
$$\emptyset=Y_{-1}\subset Y_0\subset Y_1\subset\cdots\subset Y_m=Y$$
such that $f(X_i)\subset Y_i$ for all $i$. If 
$f_*:H_j(X_i,X_{i-1})\to H_j(Y_i,Y_{i-1})$ is an isomorphism for all
$i=0,\cdots,m$ and $j\leq k+1$, then $f_*:H_j(X)\to H_j(Y)$ is an
isomorphism for $j\leq k$.
\end{prop}

\begin{proof}
This can be easily proved via spectral sequences, but we will give an
elementary proof. By induction on $p=1,2,\cdots,m+1$ we will prove that 
$$f_*:H_j(X_{i},X_{i-p})\to H_j(Y_i,Y_{i-p})$$
is an isomorphism for $p-1\leq i\leq m$ and $j\leq k$. For $p=1$
this is a hypothesis. For $p>1$ we consider the long exact sequences
of triples $(X_i,X_{i-1},X_{i-p})$ and $(Y_i,Y_{i-1},Y_{i-p})$, and
the map between them induced by $f$. 
\[\arraycolsep=0.8pt\def\arraystretch{1}
\begin{array}{ccccccccc}
H_{j+1}(X_i,X_{i-1})&\to &H_j(X_{i-1},X_{i-p})&\to
&H_j(X_i,X_{i-p}) &\to &
H_j(X_i,X_{i-1}) &\to & H_{j-1}(X_{i-1},X_{i-p})\\

\downarrow & & \downarrow & & \downarrow & & \downarrow & & \downarrow\\

H_{j+1}(Y_i,Y_{i-1})&\to &H_j(Y_{i-1},Y_{i-p})&\to
&H_j(Y_i,Y_{i-p}) &\to &
H_j(Y_i,Y_{i-1}) &\to & H_{j-1}(Y_{i-1},Y_{i-p})
\end{array}
\]
The inductive step now follows from the 5-lemma, and the Proposition
from the case $p=m+1$.
\end{proof}

\begin{prop}\label{9}
Let $(X',X)$ be a finite dimensional simplicial pair, $G'$ a group and
$G<G'$ a subgroup. Suppose that
\begin{enumerate}[(i)]
\item $G'$ acts on $X'$ without inversions,
\item $G<G'$ leaves $X$ invariant,
\item both $X,X'$ are $k$-connected,
\item every $G'$-orbit intersects $X$,
\item if two simplices of $X$ are in the same $G'$-orbit, then they are
  in the same $G$-orbit,
\item for every simplex $\sigma\in X$ the inclusion
  $Stab_G(\sigma)\hookrightarrow Stab_{G'}(\sigma)$ induces an
  isomorphism $H_j(Stab_G(\sigma))\to H_j(Stab_{G'}(\sigma))$ for
  $j\leq k-\dim\sigma$.
\end{enumerate}
Then inclusion $G\hookrightarrow G'$ induces an isomorphism $H_j(G)\to
H_j(G')$ for $j\leq k$.
\end{prop}

Note that (iv) and (v) say that the induced map $X/G\to X'/G'$ is a
homeomorphism, and (vi) says that groups associated to the ``same''
point in the orbi spaces $X/G$ and $X'/G'$ have the same homology in a range.

\begin{proof}
By (iii) we have $H_j(G)=H_j(X\times_G EG)$ and
similarly for $G'$, so it suffices to argue that the map
$$X\times_G EG\to X'\times_{G'} EG'$$
induced by $f$ is an isomorphism in $H_j$ for $j\leq k$. We will apply
Proposition \ref{1} to this map and the filtrations induced by the
skeleta; thus $(X\times_G EG)_i=X^i\times_G EG$ and similarly for the
target space. We have 
$$H_j((X\times_G EG)_i,(X\times_G EG)_{i-1})=H_j((X^i,X^{i-1})\times_G
EG)=\bigoplus_{\sigma^i\in X/G}H_{j-i}(Stab_G(\sigma^i))$$
with the similar calculation for $X'$. By (iv) and (v) both sums are over the
same set of $i$-simplices in $X/G=X'/G'$, and by (vi) homology groups
are equal.
\end{proof}

\subsection{Review of the Degree Theorem}

Here we review the Hatcher-Vogtmann Degree Theorem
\cite{MR1678155}. For a simpler proof of this theorem see
\cite{MR3177173}. Let $S\A_n$ denote the {\it spine of reduced Auter
  space in rank $n$}. This is a simplicial complex whose vertices are
basepointed marked graphs $(\Gamma,v_0,\phi)$ where:
\begin{itemize}
\item
$\Gamma$ is a finite connected graph (i.e. a 1-dimensional cell complex)
  without separating edges,
\item $v_0\in \Gamma$ is a base vertex, it has valence $>1$, and all
  other vertices have valence $>2$,
\item
$\phi:F_n\to\pi_1(\Gamma,v_0)$ is
an isomorphism (called a {\it marking}). 
\end{itemize}

Two triples $(\Gamma,v_0,\phi)$ and $(\Gamma',v_0',\phi')$ represent
the same vertex of $S\A_n$ if there is a basepoint-preserving graph
isomorphism $I:\Gamma\to\Gamma'$ with $\phi'=I_*\phi$.

We write $(\Gamma,v_0,\phi)>(\Gamma',v_0',\phi')$ if there is a forest
(subgraph with contractible components)
$F\subset\Gamma$ so that $(\Gamma',v_0',\phi')$ is equivalent to
the graph obtained from $\Gamma$ by
collapsing all components of $F$, with the induced base vertex and
marking.  

The simplicial complex $S\A_n$ is the poset of this order relation,
i.e. a simplex is an ordered chain.  It is contractible \cite{CV}. The
group $Aut(F_n)$ acts on $S\A_n$ by precomposing the marking and the
action is proper without inversions. The stabilizer of a vertex is
equal to the symmetry group of the underlying basepointed
graph. 

The {\it degree} of a basepointed graph $(\Gamma,v_0)$ is the sum
$$deg(\Gamma)=\sum_{v\neq v_0}|v|-2$$
where $|v|$ denotes the valence of $v$ and the sum runs over all
vertices of $\Gamma$ distinct from $v_0$. Denote by $S\A_{n,k+1}$ the
subcomplex of $S\A_n$ spanned by graphs with degree $\leq k+1$. This
subcomplex is $Aut(F_n)$-invariant. Note also that collapsing a forest
cannot increase the degree.

\begin{thm}  [\cite{MR1678155}]
$S\A_{n,k+1}$ is $k$-connected. 
\end{thm}

\begin{lemma}[\cite{MR1678155}]\label{roses}
\begin{enumerate}[(a)]
\item If $n>2k+2$ then every $(\Gamma,v_0)$ of degree $\leq k+1$ has a
  loop at $v_0$.
\item If $n-m+1>2k+2$ then every $(\Gamma,v_0)$ of
degree $\leq k+1$ has $m$ loops at $v_0$. 
\item If $n > 4k$ then every $(\Gamma,v_0)$ of degree
$\leq k+1$ has $2k-1$ loops at $v_0$. 
\end{enumerate}
\end{lemma}

\subsection{The stability theorem}

\begin{thm}
$H_k(Aut(F_{n}))\to H_k(Aut(F_{n+1}))$ is an isomorphism for $n>4k$.
\end{thm}

\begin{proof}
We will use Proposition \ref{9}. Let $X=S\A_{n,k+1}$,
$X'=S\A_{n+1,k+1}$ with the standard actions of $G=Aut(F_n)$ and
$G'=Aut(F_{n+1})$. We define a natural equivariant embedding
$X\hookrightarrow X'$ as follows. Write $F_{n+1}=F_n*\<t\>$ so that
$Aut(F_n)$ is identified with the subgroup of $Aut(F_{n+1})$ that
preserves $F_n$ and $t$. A vertex of $S\A_{n,k+1}$ is a triple
$(\Gamma,v_0,\phi)$ and we map it to the vertex of $S\A_{n+1,k+1}$
given by $(\Gamma\vee S^1,v_0,\phi')$. The wedge here is at the basepoint
$v_0$, and $\phi':F_n*\<t\>\to
\pi_1(\Gamma',v_0)=\pi_1(\Gamma,v_0)*\pi_1(S^1,v_0)$ is $\phi$ on the
first factor and an isomorphism on the second, and we simply
write $\phi'=\phi*id$ (there are two possible isomorphisms on the
second factor, but either
choice defines the same point in $S\A_{n+1,k+1}$).

This map on the vertices extends to a simplicial equivariant embedding
$$S\A_{n,k+1}\hookrightarrow S\A_{n+1,k+1}$$
The first three properties from Proposition \ref{9} are
clear. Property (iv) follows from Lemma \ref{roses}. E.g. start
with a vertex $(\Gamma',v_0,\phi')\in S\A_{n+1,k+1}$. Since $n+1>2k+2$ the
graph $\Gamma'$ has the form $\Gamma'=\Gamma\vee S^1$, and after
precomposing the marking, $\phi'$ has the form $\phi*id$, so the
$Aut(F_{n+1})$-orbit
of every vertex intersects $S\A_{n,k+1}$. A similar argument works for
any simplex in $S\A_{n+1,k+1}$. Say the simplex is obtained from
$(\Gamma',v_0,\phi')$ by collapsing a sequence of forests. We can
again write $\Gamma'=\Gamma\vee S^1$, change the marking, and observe
that all the forests are contained in $\Gamma$. Property (v) is easy. 

Finally, we argue (vi). We start with a vertex $(\Gamma,v_0,\phi)\in
S\A_{n,k+1}$. Write $\Gamma$ as
$\Gamma=\Gamma_0\vee R_{m}$ where $R_{m}$ denotes the wedge of
$m$ circles, and $m$ is maximal possible. According to Lemma \ref{roses},
$m\geq 2k-1$. The key point is that 
the symmetry group of $(\Gamma,v_0)$ is the direct product of symmetry
groups of $(\Gamma_0,v_0)$ and $(R_m,v_0)$, and the latter one is the
signed permutation group $S_m^{\pm}$.
Thus we have
$$Stab_{Aut(F_n)}(\Gamma,v_0,\phi)\cong D\times S_m^\pm$$
and
$$Stab_{Aut(F_{n+1})}(\Gamma\vee S^1,v_0,\phi*id)\cong D\times
S_{m+1}^\pm$$
where $D$ is the symmetry group of $(\Gamma_0,v_0)$.
So we need to argue that
$$D\times S_m^\pm\hookrightarrow D\times S_{m+1}^\pm$$
induces an isomorphism in $H_{\leq k}$. This holds for
$S_m^\pm\hookrightarrow S_{m+1}^\pm$ by Proposition \ref{6}, and in
general by the K\"unneth formula.

The argument for a simplex is similar.
\end{proof}

\begin{question}
Can the stability range be improved using the same method?
E.g. investigate what happens when theta graphs are wedged at the
basepoint, as in \cite{MR1678155}.
\end{question}

\bibliography{./ref}

\end{document}